\documentclass{amsart}
\usepackage{amsmath}
\usepackage{amsthm}
\usepackage{amssymb}
\usepackage{mathrsfs}
\usepackage{hyperref}
\usepackage{pifont}
\usepackage{tikz}
\usetikzlibrary{matrix,arrows}
\usepackage{stackengine}
\stackMath
\setstackgap{S}{1pt}
\usepackage{wasysym}

\numberwithin{equation}{section}

\newtheorem{thm}{Theorem}[section]
\newtheorem{prop}[thm]{Proposition} 

\newtheorem{lem}[thm]{Lemma}
\newtheorem{cor}[thm]{Corollary}

\theoremstyle{definition}
\newtheorem{dfn}[thm]{Definition}
\newtheorem{exmpl}[thm]{Example}
\newtheorem{?}[thm]{Question}

\theoremstyle{remark}
\newtheorem{rmk}[thm]{Remark}

\newcommand{\ds}{\displaystyle}
\newcommand{\tql}{\textquotedblleft}
\newcommand{\tqr}{\textquotedblright}
\newcommand{\noin}{\noindent}
\newcommand{\mc}{\mathcal}
\newcommand{\mb}{\mathbb}
\newcommand{\mbf}{\mathbf}
\newcommand{\ms}{\mathscr}

\newcommand{\fvna}{\textbf{fvNa}}
\newcommand{\rrz}{RR0\;}

\begin{document}
\title{Some results on tracial stability and graph products}

\author{Scott Atkinson}

\address{Vanderbilt University, Nashville, TN, USA}
\email{scott.a.atkinson@vanderbilt.edu}

\begin{abstract}
We establish the tracial stability of a certain class of graph products of $C^*$-algebras.  This result involves the development of the \tql pincushion class\tqr of finite graphs. We then apply this result in two ways. The first application yields a selective version of Lin's Theorem for almost commuting operators. The second application addresses some approximation properties of right-angled Artin groups.  In particular, we show that the full $C^*$-algebra of any right-angled Artin group is quasidiagonal and thus has a non-trivial amenable trace, and then we apply tracial stability to show when these amenable traces are in fact locally finite dimensional. 

\smallskip

\end{abstract}

\maketitle

\section*{Introduction}

The notion of stability for groups has been of significant interest in recent years.  Stability describes in a precise fashion the situation when maps on a group that approximate a homomorphism are in fact near honest homomorphisms. Thom's ICM survey \cite{thom} sheds light on various group-theoretic results, applications, and questions related to this growing subject.  Stability for groups is inextricably related to the notions of soficity and hyperlinearity. Some notable group-theoretic references include \cite{arzpau, hadshu2, deglluth, beluth, beclub}. 

This article is concerned with an analogous notion of stability for $C^*$-algebras called \emph{tracial stability} introduced by Hadwin-Shulman in \cite{hadshu}--see \S\S\ref{ts} for precise definitions and examples. In particular, we consider the question of when tracial stability is preserved by the operation of taking a graph product.  A graph product is a simultaneous generalization of free products and direct/tensor products--see \S\S \ref{gp}. We introduce the \emph{pincushion class} (cf. Definition \ref{pcdef}) of finite graphs and show in Theorem \ref{pcts} that under suitable conditions, graph products corresponding to these graphs preserve tracial stability.  This result introduces a wide class of new examples of tracially stable algebras (and stable groups) and yields two interesting applications.



In Section \ref{sellin} we apply Theorem \ref{pcts} to obtain a selective version of Lin's Theorem.  Lin's Theorem (\cite{lin2}) says that with respect to the operator norm, a pair of almost commuting self-adjoint matrices is near a pair of commuting self-adjoint matrices where the estimates can be made independent of dimension.  This shows that the condition that a pair of self-adjoints commute is stable with respect to matrices and the operator norm.  The context of Lin's Theorem has been the subject of intense study for a long time--in fact, decades preceding Lin's solution.  Mathematicians have considered many variations of this statement--some true and some false.  There is a massive amount of literature on this topic; so we only briefly give an incomplete list of relevant references: \cite{halmos, voicheis, voiccomm, davicomm, choicomm, loringcomm1, exelor, friror, hadwin, loringcomm2, filkac, glebsky, filsaf, arzpau, kacsaf, elegra}.  See \S \ref{sellin} for more discussion on Lin's Theorem. Theorem \ref{mr} gives a positive variation of Lin's Theorem for selective commuting relations with respect to the normalized Hilbert-Schmidt norm. 

In Section \ref{app} we apply Theorem \ref{pcts} to establish an approximation property for certain right-angled Artin groups--see Definition \ref{raagdfn}.  In particular, we will first show that all right-angled Artin groups have quasidiagonal full group $C^*$-algebras.  This fact guarantees the existence of (non-trivial) amenable traces (cf. Definition \ref{amendef}) on these algebras.  Our result then shows that for a right-angled Artin group $A$ coming from a graph in the pincushion class, all amenable traces on $C^*(A)$ are in fact locally finite dimensional (cf. Definition \ref{lfddef}).


\subsection*{Acknowledgments}

The author would like to thank Ben Hayes, Adrian Ioana, Mehrdad Kalantar, Jesse Peterson, David Sherman, and Pieter Spaas for helpful comments and conversations about these results.

\section{Preliminaries}\label{pre}

\subsection{Tracial stability}\label{ts} Let $\mc{U}$ denote a free ultrafilter on $\mb{N}$.  If for each $k \in \mb{N}$  $\mc{A}_k$ is a unital $C^*$-algebra and $\tau_k$ is a tracial state on $\mc{A}_k$, we let $(\mc{A}_k,\tau_k)^\mc{U}$ denote the tracial ultraproduct of the $\mc{A}_k$'s with respect to the $\tau_k$'s. Let $a_k \in \mc{A}_k$ for each $k \in \mb{N}$; then we let $(a_k)_\mc{U} \in (\mc{A}_k, \tau_k)^\mc{U}$ denote the coset of the sequence $(a_k)_{k \in \mb{N}}$ in $\ds \prod_{k \in \mb{N}} \mc{A}_k$. We will sometimes suppress the $\tau_k$ notation when the context is clear. See Appendix A of \cite{brownozawa} for an introduction to ultrafilters and ultraproducts.  

\begin{dfn}[\cite{hadshu}]
Let $\mc{A}, \mc{A}_k$ be unital $C^*$-algebras, and for each $k \in \mb{N}$ let $\tau_k$ be a tracial state on $\mc{A}_k$.  A unital $*$-homomorphism $\pi: \mc{A} \rightarrow (\mc{A}_k,\tau_k)^\mc{U}$ is \emph{approximately liftable} if there is a set $E \in \mc{U}$ such that for each $k \in E$ there is a unital $*$-homomorphism $\pi_k: \mc{A} \rightarrow \mc{A}_k$ such that $\pi(a) = (\pi_k(a))_\mc{U}$ for every $a \in \mc{A}$ where $\pi_k(a) = 0$ for $k \notin E$.  Such a sequence $\left\{\pi_k\right\}$ is called a \emph{lifting} of $\pi$.
\end{dfn}

\noin The following is a general lemma about approximate liftings that will be a valuable utility.  It first appeared as Lemma 2.2 in \cite{hadshu}.

\begin{lem}[\cite{hadshu}]\label{techlemma}
Let $\mc{A}$ be a separable unital $C^*$-algebra with generators $\left\{b_j\right\}_{j \in \mb{N}}$.  For each $k \in \mb{N}$, let $\mc{A}_k$ be a unital $C^*$-algebra and $\tau_k$ be a tracial state on $\mc{A}_k$.  A $*$-homomorphism $\pi: \mc{A} \rightarrow (\mc{A}_k,\tau_k)^\mc{U}$ with $\pi(b_j) = (b_j(k))_\mc{U}$ is approximately liftable if and only if for every $\varepsilon > 0$ and every $N \in \mb{N}$, there is a set $E \in \mc{U}$ with the property that for every $k \in E$ there is a unital $*$-homomorphism $\pi_k: \mc{A} \rightarrow \mc{A}_k$ such that for every $1 \leq j \leq N$ and every $k \in E$, \[||\pi_k(b_j) - b_j(k)||_{2,\tau_k} < \varepsilon\] where $||\cdot ||_{2,\tau_k}$ denotes the Hilbert-Schmidt norm coming from the trace $\tau_k$. 
\end{lem}

\begin{dfn}[\cite{hadshu}]
Let $\ms{C}$ denote a class of unital $C^*$-algebras closed under $*$-isomorphism.  A unital $C^*$-algebra $\mc{A}$ is \emph{$\ms{C}$-tracially stable} if every unital $*$-homomorphism $\pi: \mc{A} \rightarrow (\mc{A}_k, \tau_k)^\mc{U}$ is approximately liftable whenever $\mc{A}_k \in \ms{C}$ and $\tau_k$ is a tracial state on $\mc{A}_k$ for every $k \in \mb{N}$.
\end{dfn}


Next, we discuss some results that give some examples of tracially stable $C^*$-algebras.  Recall that a $C^*$-algebra is \emph{real rank zero} (RR0) if its self-adjoint elements can be approximated by self-adjoint elements with finite spectrum.

\begin{thm}[\cite{hadshu}]\label{commts}
Let $\ms{C}$ be a class of \rrz $C^*$-algebras. Then any separable unital commutative $C^*$-algebra $C(X)$ is $\ms{C}$-tracially stable. 
\end{thm}

Let \textbf{II}$_1$ denote the class of II$_1$-factor von Neumann algebras. Hadwin-Shulman studied the tracial stability of finitely presented groups in \cite{hadshu2}; the main result of that paper is the following theorem.

\begin{thm}[\cite{hadshu2}]
If $G$ is a one-relator group with nontrivial center, then $C^*(G)$ is \textbf{II$_1$}-tracially stable.
\end{thm}

A $C^*$-algebra $\mc{A}$ is \emph{tracially amenable} if for every tracial state $\tau$ on $\mc{A}$, $\pi_\tau(\mc{A})''$ is hyperfinite where $\pi_\tau$ denotes the GNS representation induced by $\tau$.   The following theorem first appeared as Theorem 4.10 of \cite{hadwin-li} for nuclear algebras.

\begin{thm}[\cite{hadwin-li, saa}]\label{tam}
Any unital separable tracially amenable $C^*$-algebra is \textbf{II$_1$}-tracially stable.
\end{thm}
\noin See Theorem 5.15 of \cite{saa} for the tracially amenable case and a concise proof.  We can use tracial stability to characterize the separably acting hyperfinite II$_1$-factor in the following way.

\begin{thm}\label{Rchar}
Let $N$ be a separably acting II$_1$-factor satisfying the Connes Embedding Problem. That is, there exists a unital trace preserving embedding $N \rightarrow R^\mc{U}$ where $R$ denotes the separably acting hyperfinite II$_1$-factor.  Then $N$ is hyperfinite if and only if $N$ is \textbf{II}$_1$-tracially stable.
\end{thm}

\begin{proof}
$(\mbf{\Rightarrow})$: This follows from Theorem \ref{tam}.

$(\mbf{\Leftarrow})$:  If $N$ is not hyperfinite, then it does not embed into $R$.  Thus no embedding $N \rightarrow R^\mc{U}$ is approximately liftable.
\end{proof}

\noin This result leads to the following open question.

\begin{?}\label{21ts}
In \cite{jung}, Jung showed that a (finitely generated) separably acting II$_1$-factor $N$ satisfying the Connes Embedding Problem is hyperfinite if and only if any two embeddings $N \rightarrow R^\mc{U}$ are unitarily equivalent. Theorem \ref{Rchar} shows that the separably acting hyperfinite II$_1$-factor $R$ has the weaker property that any embedding $R \rightarrow R^\mc{U}$ is approximately liftable.  Does this property characterize $R$?  That is, if $N$ is a separably acting II$_1$-factor  such that every embedding $N \rightarrow N^\mc{U}$ is approximately liftable, then is $N$ necessarily hyperfinite?
\end{?}

The next result follows directly from the universal property of free products.

\begin{prop}\label{freeprop}
Let $\ms{C}$ be any class of unital $C^*$-algebras closed under \linebreak $*$-isomorphisms. The full/universal free product of finitely many $\ms{C}$-tracially stable $C^*$-algebras is $\ms{C}$-tracially stable.
\end{prop}

\noin The corresponding tensor product result takes more care.	
	
\begin{thm}[\cite{hadshu}]\label{tensorthm}
Let $\ms{C}$ be a class of \rrz $C^*$-algebras closed under taking direct sums and unital corners. If $\mc{A}$ is $\ms{C}$-tracially stable and $X$ is a compact Hausdorff space, then $\mc{A} \otimes C(X)$ is $\ms{C}$-tracially stable. 
\end{thm}

\noin Theorem \ref{tensorthm} yields an interesting corollary about relative commutants of tracially stable subalgebras of ultraproducts which may be known to experts but has not appeared in the literature. 

\begin{cor}
Let $\ms{C}$ be a class of \rrz $C^*$-algebras closed under taking direct sums and unital corners. Let $\mc{A}$ be a unital $\ms{C}$-tracially stable $C^*$-algebra.  For each $k \in \mb{N}$, let $\mc{A}_k$ be a member of $\ms{C}$ and $\tau_k$ be a tracial state on $\mc{A}_k$.  Fix a $*$-homomorphism $\pi: \mc{A} \rightarrow (\mc{A}_k,\tau_k)^\mc{U}$.  

\begin{enumerate}
\item If $x \in \pi(\mc{A})' \cap (\mc{A}_k,\tau_k)^\mc{U}$ is normal (self-adjoint, a projection, a unitary), then there is a lifting $\left\{\pi_k: \mc{A} \rightarrow \mc{A}_k\right\}$ of $\pi$ such that $x \in (\pi_k(\mc{A})'\cap \mc{A}_k, \tau_k)^\mc{U}$.  That is, $x = (x_k)_\mc{U}$ where $x_k \in \pi_k(\mc{A})'\cap \mc{A}_k$ is normal (respectively self-adjoint, a projection, a unitary) for every $k \in E$ for some $E \in \mc{U}$.  

\item From (1) we have \[ \pi(\mc{A})'\cap (\mc{A}_k,\tau_k)^\mc{U} = \bigvee_{\left\{\pi_k\right\} \text{is a lifting of } \pi} (\pi_k(\mc{A})'\cap \mc{A}_k, \tau_k)^\mc{U}\] where $\bigvee$ denotes the join.
\end{enumerate}
\end{cor}

\begin{proof}
\noin (1): We prove the normal case; the other cases follow \emph{mutatis mutandis}. Let $x \in \pi(\mc{A})' \cap (\mc{A}_k,\tau_k)^\mc{U}$ be normal, and let $X$ denote the spectrum of $x$.  By Theorem \ref{tensorthm}, $\mc{A} \otimes C(X)$ is $\ms{C}$-tracially stable.  Let $\rho: C(X) \rightarrow (\mc{A}_k,\tau_k)^\mc{U}$ denote the $*$-homomorphism given by $\text{id} \mapsto x$.  By hypothesis, the images of $\pi$ and $\rho$ commute. Thus we can form the approximately liftable $*$-homomorphism $\pi \otimes \rho: \mc{A} \otimes C(X) \rightarrow (\mc{A}_k,\tau_k)^\mc{U}$.  So there is a lifting $\left\{\theta_k\right\}$ of $\pi \otimes \rho$ such that for every $y \in \mc{A} \otimes C(X), (\pi \otimes \rho)(y) = (\theta_k(y))_\mc{U}$.  Put $\pi_k := \theta_k|_{\mc{A}\otimes 1}$.  Then $\left\{\pi_k\right\}$ is a lifting of $\pi$ and $x = \rho(\text{id}) = (\theta_k(1\otimes \text{id}))_\mc{U} \in (\pi_k(\mc{A})'\cap \mc{A}_k, \tau_k)^\mc{U}.$

(2):  The algebra $\pi(\mc{A})' \cap (\mc{A}_k,\tau_k)^\mc{U}$ is spanned by its self-adjoints.  By (1), for every self-adjoint $x \in \pi(\mc{A})' \cap (\mc{A}_k,\tau_k)^\mc{U}$, there is a lifting $\left\{\pi_k\right\}$ such that $x \in (\pi_k(\mc{A})'\cap \mc{A}_k,\tau_k)^\mc{U}$.  Thus \[\pi(\mc{A})' \cap (\mc{A}_k,\tau_k)^\mc{U} \subset \bigvee_{\left\{\pi_k\right\} \text{is a lifting of } \pi} (\pi_k(\mc{A})'\cap \mc{A}_k, \tau_k)^\mc{U}.\]  The reverse containment is obvious.
\end{proof}

\begin{rmk}
Some concrete classes of \rrz $C^*$-algebras of interest are \textbf{fvNa}: finite von Neumann algebras, \textbf{ff}: finite factor von Neumann algebras, \textbf{II$_1$}: II$_1$-factors, and \textbf{matrix}: matrix algebras. The class \fvna\; is a class of \rrz $C^*$-algebras closed under taking direct sums and unital corners.  This class contains matrix algebras and II$_1$-factors.  So any algebra that is $\ms{C}$-tracially stable for any class $\ms{C}$ of \rrz $C^*$-algebras closed under taking direct sums and unital corners is \fvna-, \textbf{ff}-, \textbf{II$_1$}-, and \textbf{matrix}-tracially stable.
\end{rmk}

The following proposition relates a notion of group stability with \textbf{matrix}-tracial stability.

\begin{prop}\label{gpprop}
Let $G$ be a discrete group.  Using the notation of \cite{beclub}, $G$ is HS-stable if and only if the full group $C^*$-algebra is \textbf{matrix}-tracially stable.
\end{prop}

\subsection{Graph products}\label{gp} In operator algebras, graph products unify the notions of free and tensor products.  In particular, given a simplicial graph $\Gamma$ assign an algebra to each vertex. If there is an edge between two vertices then the two corresponding algebras commute with each other in the graph product; if there is no edge between two vertices then the two corresponding algebras have no relations with each other within the graph product.  Thus free products are given by edgeless graphs, and tensor products are given by complete graphs.  Such products initially appeared in the group theory context, and the most well-known examples are right-angled Artin groups (graph products of $\mb{Z}$) and right-angled Coxeter groups (graph products of $\mb{Z}/2\mb{Z}$).  See \cite{baudisch, chiswell, droms1, droms2, droms3, green, charney, wise}. Interest in graph products has recently entered the realm of operator algebras.  See \cite{mlot,spewys,casfim,rec,li2,gpucp,gphaag}.

\begin{dfn}
A \emph{simplicial graph} $\Gamma$ is a graph with the following properties
\begin{enumerate}
\item undirected; 
\item no single-vertex loops;
\item there is at most one edge between vertices.
\end{enumerate}
We denote the set of vertices and the set of edges of $\Gamma$ as $V\Gamma$ and $E\Gamma$ respectively. We can consider $E\Gamma$ as a symmetric subset of $V\Gamma \times V\Gamma$ that does not intersect the diagonal, and when convenient, we will identify a pair $(v,w)\in E\Gamma$ with $(w,v)$.
\end{dfn}

Per usual, there are reduced and universal versions of graph products of $C^*$-algebras.  In this article, we will consider the universal graph product of $C^*$-algebras defined as follows.

\begin{dfn}\label{unidef}
Given a simplicial graph $\Gamma$ and unital $C^*$-algebras $\mc{A}_v$ for every $v \in V\Gamma$, the \emph{universal graph product $C^*$-algebra} is the unique unital $C^*$-algebra $\bigstar_\Gamma \mc{A}_v$ together with unital $*$-homomorphisms $\iota_v: \mc{A}_v \rightarrow \bigstar_\Gamma \mc{A}_v$  satisfying the following universal property.
\begin{enumerate}
\item $[\iota_v(a),\iota_w(b)] = 0$ whenever $a \in \mc{A}_v, b \in \mc{A}_w, (v,w) \in E\Gamma$;
\item For any unital $C^*$-algebra $\mc{B}$ with $*$-homomorphisms $\left\{f_v: \mc{A}_v \rightarrow \mc{B}\right\}_{v \in V\Gamma}$ such that $[f_v(a),f_w(b)] = 0$ whenever $a \in \mc{A}_v, b \in \mc{A}_w, (v,w) \in E\Gamma$, there exists a unique $*$-homomorphism $\bigstar_\Gamma f_v: \bigstar_\Gamma \mc{A}_v \rightarrow \mc{B}$ such that $\bigstar_\Gamma f_v \circ \iota_{v_0} = f_{v_0}$ for every $v_0 \in V\Gamma$.
\end{enumerate}
See \S 3 of \cite{mlot}, Remark 2.5 of \cite{casfim}, and \S\S 1.1 of \cite{gpucp} for discussions on the existence of such algebras.
\end{dfn}

\begin{dfn}
Let $\Gamma$ be a simplicial graph, and let $\mc{A}$ be a $C^*$-algebra.  For each $v \in V\Gamma$, let $\mc{A}_v$ be a $C^*$-subalgebra of $\mc{A}$.  If $[a,b] = 0$ whenever $a \in \mc{A}_v$ and $b \in \mc{A}_w$ with $(v,w) \in E\Gamma$, then we say the subalgebras $\left\{\mc{A}_v\right\}_{v \in V\Gamma}$ \emph{commute according to $\Gamma$}.  If $\left\{A_v\right\}_{v\in V\Gamma}$ are elements of $\mc{A}$ indexed by $V\Gamma$, then the elements $\left\{A_v\right\}_{v\in V\Gamma}$ commute according to $\Gamma$ if $[A_v,A_w] = 0$ whenever $(v,w) \in E\Gamma$.
\end{dfn}

\noin We can define the graph product of groups using a universal property analogous to the $C^*$-algebraic one in Definition \ref{unidef}.  Alternatively, one can define a graph product of groups as follows. 

\begin{dfn}
Fix a simplicial (undirected, no single-vertex loops, at most one edge between vertices) graph $\Gamma$. For each $v \in V\Gamma$, let $G_v$ be a group. The \emph{graph product group} $\bigstar_\Gamma G_v$ is given by the free product $*_{v\in V} G_v$ modulo the relations $[g,h] =1$ whenever $g \in G_v, h \in G_w$ and $(v,w) \in E\Gamma$. 
\end{dfn}

For the remainder of this section, fix a simplicial graph $\Gamma$, and for each $v \in V\Gamma$, let $\mc{A}_v$ be a unital $C^*$-algebra. When working with graph products, the bookkeeping can be done by considering words with letters from the vertex set $V\Gamma$.  Such words are given by finite sequences of elements from $V\Gamma$ and will be denoted with bold letters.  In order to encode the commuting relations given by $\Gamma$, we consider the equivalence relation generated by the following relations.
\begin{align*}
(v_1,\dots, v_i, v_{i+1},\dots, v_n) &\sim (v_1,\dots, v_i, v_{i+2}, \dots, v_n) &\text{if} &&v_i = v_{i+1}\\
(v_1,\dots, v_i, v_{i+1},\dots, v_n) &\sim (v_1,\dots, v_{i+1}, v_i,\dots, v_n) &\text{if} &&(v_i,v_{i+1}) \in E\Gamma.
\end{align*}
The concept of a reduced word is central to the theory of graph products.  The following definition is Definition 3.2 of \cite{spenic} in graph language; the equivalent definition in \cite{casfim} appears differently.

\begin{dfn}\label{redv}
A word $\mbf{v} = (v_1,\dots,v_n)$ is \emph{reduced} if whenever $v_k = v_l, k < l$, then there exists a $p$ with $k< p < l$ such that $(v_k, v_p) \notin E\Gamma$.  Let $\mc{W}_\text{red}$ denote the set of all reduced words.  We take the convention that the empty word is reduced.
\end{dfn}

\begin{prop}[\cite{green}, \cite{casfim}]\label{reducedlemma}\hspace*{\fill}
\begin{enumerate}
\item Every word $\mbf{v}$ is equivalent to a reduced word $\mbf{w} = (w_1,\dots, w_n)$.  (We let $|\mbf{w}|=n$ denote the \emph{length} of the reduced word.)
\item If $\mbf{v} \sim \mbf{w}\sim\mbf{w}'$ with both $\mbf{w}$ and $\mbf{w}'$ reduced, then the lengths of $\mbf{w}$ and $\mbf{w}'$ are equal and $\mbf{w}' = (w_{\sigma(1)},\dots,w_{\sigma(n)})$ is a permutation of $\mbf{w}$.  Furthermore, this permutation $\sigma$ is unique if we insist that whenever $w_k = w_l, k< l$ then $\sigma(k) < \sigma(l)$.
\end{enumerate}
\end{prop}

%


\section{Lifting selective commuting relations}\label{lscr}

\subsection{The pincushion class}\label{pincushion} We begin this section by defining the \emph{pincushion class} of graphs to which our main result applies.  Let $\ms{G}$ denote the collection of all finite simplicial graphs.  

\begin{dfn}Let $\ms{G}^{(1)}$ denote the class of finite simplicial graphs where each $\Gamma \in \ms{G}^{(1)}$ is obtained by recursively adding a vertex and adhering to one of the following two rules.
\begin{enumerate}
\item The new vertex is isolated;
\item The new vertex is adjacent to exactly one vertex.
\end{enumerate}
\end{dfn}

\begin{dfn}
Given $\Gamma \in \ms{G}$, a \emph{pin} of $\Gamma$ is a vertex $v \in V$ that is adjacent to exactly one other vertex $w$.  That is, there is a unique $w \in V\Gamma$ such that $(v,w) \in E\Gamma$.  Let $P\Gamma$ denote the (possibly empty) set of all pins of $\Gamma$.
\end{dfn}

\noin So rule (2) above describes forming a pin. For example, the following graph is contained in $\ms{G}^{(1)}$, and the gray vertices are pins of the graph.

\vspace{.2cm}
\begin{center}
\begin{tikzpicture}
\path (-3,0) node [draw,shape=circle,fill=gray] (p0) {}
(-1,0) node [draw,shape=circle,fill=black]  (p1) {}
(1,0) node [draw,shape=circle,fill=black] (p2) {}
(3,0) node [draw,shape=circle,fill=gray] (p3){};
\draw (p0) -- (p1)
(p2) -- (p3)
(p1) -- (p2);
\end{tikzpicture}
\end{center}
\vspace{.2cm}

To define the pincushion class of graphs, we generalize the construction of $\ms{G}^{(1)}$ by introducing the following \tql pinning\tqr operation on $\ms{G}$.  

\begin{dfn}Given $\Gamma \in \ms{G}$, fix $v \in V\Gamma$ and define $\Phi_{(\Gamma, v)}: \ms{G} \rightarrow \ms{G}$ as follows. For $\Gamma' \in \ms{G}$, put
\begin{align*}
V\Phi_{(\Gamma, v)}(\Gamma') &:= V\Gamma \sqcup V\Gamma'\\
E\Phi_{(\Gamma,v)}(\Gamma') &:= E\Gamma \sqcup E\Gamma' \sqcup \left\{(w,v): w \in V\Gamma'\right\}
\end{align*}
We call this \emph{pinning $\Gamma'$ to $\Gamma$ at $v$}.
\end{dfn}

\begin{exmpl}
We illustrate this pinning operation in this example.  Let $\Gamma$ be given by the following graph, and let $v \in V\Gamma$ be as indicated below.

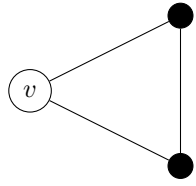
\begin{center}
\begin{tikzpicture}
\path (1,0) node [draw,shape=circle,] (p2) {$v$}
(3,1) node [draw,shape=circle,fill=black] (p3){}
(3,-1) node [draw,shape=circle,fill=black] (p4){};
\draw 
(p2) -- (p3)
(p2) -- (p4)
(p3) -- (p4);
\end{tikzpicture}
\end{center}

Let $\Gamma'$ be given by the following graph.

\begin{center}
\begin{tikzpicture}
\path (0,-1) node [draw,shape=circle,fill=gray] (p0) {}
(0,0) node [draw,shape=circle,fill=gray]  (p1) {}
(0,1) node [draw,shape=circle,fill=gray] (p2) {};
\draw (p0) -- (p1)
(p1) -- (p2);
\end{tikzpicture}
\end{center}

Then $\Phi_{(\Gamma,v)}(\Gamma')$ is given as follows.

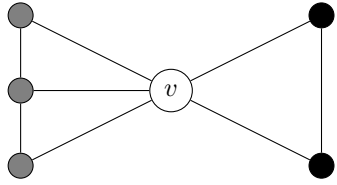
\begin{center}
\begin{tikzpicture}
\path (-1,-1) node [draw,shape=circle,fill=gray] (q0) {}
(-1,0) node [draw,shape=circle,fill=gray]  (q1) {}
(-1,1) node [draw,shape=circle,fill=gray] (q2) {}
(1,0) node [draw,shape=circle,] (p2) {$v$}
(3,1) node [draw,shape=circle,fill=black] (p3){}
(3,-1) node [draw,shape=circle,fill=black] (p4){};
\draw (q0) -- (p2)
(q1) -- (p2)
(q2) -- (p2)
(q0) -- (q1)
(q1) -- (q2)
(p2) -- (p3)
(p2) -- (p4)
(p3) -- (p4);
\end{tikzpicture}
\end{center}

\end{exmpl}

Using this pinning operation, we recursively construct the classes $\ms{G}^{(m)}$ for $m \in \mb{N}$.  

\begin{dfn}\label{gm}
Assume $\ms{G}^{(m)}$ has been formed; define $\ms{G}^{(m+1)}$ as follows.  A graph $\Gamma \in \ms{G}^{(m+1)}$ is obtained by recursively appending a graph from $\ms{G}^{(m)}$ and adhering to one of the following two rules.

\begin{enumerate}
\item The appended graph is isolated from the graph obtained in the previous step;
\item The appended graph is pinned to the graph from the previous step at some vertex.
\end{enumerate}
\end{dfn}

\begin{rmk}
We observe here that if we define $\ms{G}^{(0)}$ to be the set of finite simplicial graphs consisting only of the single vertex graph, then the above recursive definition for $m = 0$ recovers the definition of $\ms{G}^{(1)}$.  So the collections $\ms{G}^{(m)}$ can be considered as graded generalizations of $\ms{G}^{(1)}$ via the pinning operation.
\end{rmk}

Note that $\ms{G}^{(m-1)} \subset \ms{G}^{(m)}$, and the following proposition shows that this containment is strict.

\begin{prop}
For every $m \in \mb{N}$, $\ms{G}^{(m)} \setminus \ms{G}^{(m-1)} \neq \emptyset.$ In particular, if $K_{m+1}$ denotes the complete graph with $m+1$ vertices, then $K_{m+1} \in \ms{G}^{(m)}\setminus \ms{G}^{(m-1)}$.
\end{prop}
\begin{proof}
We proceed by induction on $m$.  The base case of $m=1$ is clear.  Let $m >1$ and assume that for $1 \leq k < m$, we have $K_{k+1} \in \ms{G}^{(k)}\setminus \ms{G}^{(k-1)}$. Since $K_{m+1} = \Phi_{(K_1,v_1)}(K_m)$ where $VK_1 = \left\{v_1\right\}$ and $K_m \in \ms{G}^{(m-1)}$ by the induction hypothesis, we have that $K_{m+1} \in \ms{G}^{(m)}$. 

Suppose for the sake of contradiction that $K_{m+1} \in \ms{G}^{(m-1)}$. Since every pair of vertices in $K_{m+1}$ is adjacent, $K_{m+1}$ is obtained by using rule (2) of Definition \ref{gm} exclusively.  Let $v_0 \in VK_{m+1}$ denote the vertex of $K_{m+1}$ to which a graph $\Gamma_0 \in \ms{G}^{(m-2)}$ was pinned in the last step of the construction of $K_{m+1}$.  The graph $\Gamma_0$ is a subgraph of $K_{m+1}$ so it must be complete too. Since $K_m \notin \ms{G}^{(m-2)}$ by the induction hypothesis, then $\Gamma_0 \neq K_m$.  Thus there must be a vertex $v_1 \in VK_{m+1}$ different from $v_0$ such that $v_1 \notin V\Gamma_0$.  Hence $v_1$ must have been present in the construction process before $\Gamma_0$ was pinned to $v_0$.  But this means that no vertex in $\Gamma_0$ is adjacent to $v_1$, contradicting the completeness of $K_{m+1}$.  Thus $K_{m+1} \notin \ms{G}^{(m-1)}$.
\end{proof}

%

We are now prepared to define the pincushion class of finite simplicial graphs.

\begin{dfn}\label{pcdef}
The \emph{pincushion class} of finite simplicial graphs is denoted $\ms{G}^{(\infty)}$ and is given by \[\ms{G}^{(\infty)} := \bigcup_{m \in \mb{N}} \ms{G}^{(m)}.\]
\end{dfn}

\noin In addition to containing all complete graphs, the pincushion class also contains all finite trees.  This is a straightforward exercise left to the reader.

\subsection{Main result}\label{mainresult}

Proposition \ref{freeprop} and Theorem \ref{tensorthm} show that in certain circumstances, tracial stability is preserved under free products in general and tensor products with commutative $C^*$-algebras.  When a property behaves well with free and tensor products, it is natural to ask about how it behaves with graph products.  If the graph comes from the pincushion class and certain component algebras are commutative, then the corresponding graph product indeed preserves tracial stability.  For the remainder of this section, let $\ms{C}$ be a class unital \rrz $C^*$-algebras closed under taking direct sums and unital corners. The following theorem can be considered as a generalization of Theorem \ref{tensorthm}.   

\begin{thm}\label{pcts}
Let $\Gamma \in \ms{G}^{(\infty)}$ be a pincushion class graph.  For each $v \in V\Gamma$, let $\mc{A}_v$ be a separable unital $C^*$-algebra satisfying the following properties.
\begin{enumerate}
\item If $v \in V\Gamma$ is a pin not adjacent to another pin, then $\mc{A}_v$ is $\ms{C}$-tracially stable;
\item If $v \in V\Gamma$ is isolated, then $\mc{A}_v$ is $\ms{C}$-tracially stable.
\item If $v \in V\Gamma$ fits neither of the above descriptions, then $\mc{A}_v = C(X_v)$ for some compact Hausdorff space $X_v$.
\end{enumerate}
Then $\bigstar_\Gamma \mc{A}_v$ is $\ms{C}$-tracially stable.
\end{thm}

\begin{proof}
Hadwin-Shulman's proof of Theorem \ref{tensorthm} in \cite{hadshu} serves as the inspiration for this proof.  

We proceed by induction on $|V\Gamma |$.  The base case follows tautologically. Let $K >1$, and suppose that the theorem holds for $|V\Gamma| < K$.  By the induction hypothesis and Proposition \ref{freeprop}, the result clearly holds if $\Gamma$ is not connected.  Assume that $\Gamma$ is connected.  

\vspace{.2cm}

\noin \textbf{Case I}: $P\Gamma = \emptyset$.  There is a construction of $\Gamma$ in which the last step is to pin a subgraph $\Gamma_1 \in \ms{G}^{(\infty)}$ with $|V\Gamma_1| > 1$ to a vertex $v_1$.  By the induction hypothesis, $\bigstar_{\Gamma_1}\mc{A}_w$ is $\ms{C}$-tracially stable.  Consider the graph $\Gamma'$ given by taking $\Gamma$ and collapsing $\Gamma_1$ to a single vertex $v'$ with an edge connecting $v'$ to $v_1$. Thus, $v'$ is a pin of $\Gamma'$.  For each $v \in V\Gamma'$, put \[\mc{A}'_v := \left\{\begin{array}{lcr}
\mc{A}_v & \text{if} & v \neq v'\\
\bigstar_{\Gamma_1}\mc{A}_w & \text{if} & v = v'.
\end{array}\right.\] Then by the induction hypothesis, we have \[\bigstar_\Gamma \mc{A}_v = \bigstar_{\Gamma'} \mc{A}'_v\] is $\ms{C}$-tracially stable.

\noin \textbf{Case II}: $P\Gamma \neq \emptyset$.  Let $v_0 \in P\Gamma$ be a pin of $\Gamma$.  Let $v_1$ be the (unique) vertex adjacent to $v_0$.  If $v_1$ is also a pin, then the subgraph $\Gamma_{01}$ formed by $v_0, v_1$, and the edge joining them is isolated from the complementary subgraph $\Gamma_c$ of $\Gamma$.  Then \[\bigstar_\Gamma \mc{A}_v = (\bigstar_{\Gamma_{01}}\mc{A}_v) * (\bigstar_{\Gamma_c}\mc{A}_v) = (C(X_{v_0}) \otimes C(X_{v_1})) * (\bigstar_{\Gamma_c} \mc{A}_v)\] is $\ms{C}$-tracially stable by Theorem \ref{commts}, Proposition \ref{freeprop}, and the induction hypothesis.  

If $v_1$ is not a pin, then $\mc{A}_{v_0}$ is $\ms{C}$-tracially stable, and $\mc{A}_{v_1} = C(X_{v_1})$ for some compact Hausdorff $X_{v_1}$. For each $k \in \mb{N}$ let $M_k$ be a $C^*$-algebra in $\ms{C}$, and let $\tau_k$ be a tracial state on $M_k$.  Fix a unital $*$-homomorphism $\pi: \bigstar_\Gamma \mc{A}_v \rightarrow (M_k,\tau_k)^\mc{U}$.  As in the proof of Theorem \ref{tensorthm} in \cite{hadshu}, we will use Lemma \ref{techlemma} to show that $\pi$ is approximately liftable. Fix $\varepsilon > 0$ and contractions $x_1,\dots x_n \in \bigstar_\Gamma \mc{A}_v$.  We set out to find $*$-homomorphisms $\pi_k: \bigstar_\Gamma \mc{A}_v \rightarrow M_k$ such that 
\begin{align}
||(\pi_k(x_j))_\mc{U} - \pi(x_j)||_{2,\mc{U}} < 3\varepsilon \label{tol}
\end{align}
for every $1 \leq j \leq n$.  For each $1 \leq j \leq n$ there are elements $w_i^{(j)} \in \bigstar_\Gamma \mc{A}_v$ of the form 
\begin{align}
w_i^{(j)} = a(j,i,1)\cdots a(j,i,s_i^{(j)})\label{decomp}
\end{align} where $a(j,i,l) \in (\mc{A}_{v(j,i,l)})_{\leq 1}$ for $1 \leq l \leq s_i^{(j)}$ and $(v(j,i,1),\cdots, v(j,i,s_i^{(j)}))$ is reduced with the property that \[\left\|x_j - \sum_{i=1}^{N_j} w_i^{(j)}\right\| < \varepsilon\] (cf. \S\S 1.1 in \cite{gpucp}).    
The image $\pi(\mc{A}_{v_1}) = \pi(C(X_{v_1}))$ is commutative, so let $\pi(C(X_{v_1})) = C(\Omega)$ for some compact Hausdorff $\Omega$.  Let $a_1(j_1,i_1,l_1),\dots a_1(j_T,i_T,l_T)$ be the elements of $\mc{A}_{v_1}$ appearing in the above decompositions of the $w_i^{(j)}$'s, and for each $1 \leq j \leq n, 1 \leq i \leq N_j$ let $S_i^{(j)}$ denote the number of elements in $\left\{a_1(j_1,i_1,l_1),\dots a_1(j_T,i_T,l_T)\right\}$  appearing in the above decomposition of $w_i^{(j)}$. Put 
\begin{align*}
N&:= \max_{1 \leq j \leq n} N_j\\
&\text{and}\\
S&:= \max_{\substack{1 \leq j \leq n \\ 1\leq i \leq N_j}} S_i^{(j)}.
\end{align*}
We can approximate each of $a_1(j_1,i_1,l_1),\dots a_1(j_T,i_T,l_T)$ using simple functions.  That is, there is a disjoint collection $\left\{E_1,\dots, E_m\right\}$ of Borel subsets of $\Omega$ whose union is $\Omega$, and for each $1 \leq d \leq m$, there is an element $\omega_d \in E_d$ such that 
\begin{align}
\left\| \pi(a_1(j_t,i_t,l_t)) - \sum_{d=1}^m \pi(a_1(j_t,i_t,l_t))(\omega_d)\chi_{E_d}\right\| < \frac{\varepsilon}{NS}\label{est}
\end{align} for $1 \leq t \leq T$.

 Let $\Gamma_2$ denote the subgraph of $\Gamma$ given by $V\Gamma_2 = V\Gamma \setminus\left\{v_0\right\}$ and $E\Gamma_2 = E\Gamma \setminus \left\{(v_0,v_1)\right\}$.  
 For each $v \in V\Gamma_2$, let $\mc{B}_v$ be defined as follows. \[\mc{B}_v := \left\{\begin{array}{lcr}
\mc{A}_v & \text{if} & v \neq v_1\\
C^*(\chi_{E_1},\cdots \chi_{E_d}) & \text{if} & v = v_1
\end{array}\right.\]
It is clear that $\Gamma_2 \in \ms{G}^{(\infty)}$. Indeed, in the construction of $\Gamma \in \ms{G}^{(\infty)}$, after $v_1$ appears, the step of pinning $v_0$ to $v_1$ is independent of the rest of the steps in the construction.  So by the induction hypothesis, $\bigstar_{\Gamma_2} \mc{B}_v$ is $\ms{C}$-tracially stable. Let $\rho=\bigstar_{\Gamma_2}\rho_v: \bigstar_{\Gamma_2}\mc{B}_v \rightarrow (M_k,\tau_k)^\mc{U}$ where 
\[\rho_v = \left\{\begin{array}{lcr}
\pi|_v & \text{if} & v \neq v_1\\
\text{id} & \text{if} &v = v_1.
\end{array}\right.\] 
Note that $\rho$ is well-defined because $C^*(\chi_{E_1},\dots ,\chi_{E_d}) \subset W^*(\pi(C(X_{v_1})))$, thus $\left\{\rho_v(\mc{B}_v)\right\}_{v \in V\Gamma_2}$ satisfies all requisite commuting relations. Hence by $\ms{C}$-tracial stability, for each $k \in \mb{N}$, we can find a unital $*$-homomorphism $\rho_k: \bigstar_{\Gamma_2} \mc{B}_v \rightarrow M_k$ such that $(\rho_k(a))_\mc{U} = \rho(a)$ for every $a \in \bigstar_{\Gamma_2}\mc{B}_v$, and for every $v \in V\Gamma_2$, let $\rho_{v,k} = \rho_k|_{\mc{B}_v}$. 

We now construct the desired sequence of $*$-homomorphisms $\left\{\pi_k\right\}$ satisfying \eqref{tol}. This follows an argument parallel to the one found in the proof of Theorem 2.7 of \cite{hadshu}. For each $1 \leq d \leq m$ and each $k \in \mb{N}$ let $P_{d,k}:= \rho_{v_1,k}(\chi_{E_d})$. Then we have $\chi_{E_d} = (P_{d,k})_\mc{U}$ and $\left\{P_{1,k},\dots P_{m,k}\right\}$ is a pairwise orthogonal partition of unity in $M_k$.  Observe that \[C^*(\chi_{E_1},\dots, \chi_{E_m}) \subset \pi(C^*(\mc{A}_{v_0}, C(X_{v_1})))'\cap (M_k,\tau_k)^\mc{U}.\] Therefore \[\pi(C^*(\mc{A}_{v_0}, C(X_{v_1}))) \subset \left(\sum_{d=1}^m P_{d,k}M_kP_{d,k}, \tau_k \right)^\mc{U} = \bigoplus_{d=1}^m\left(\left(P_{d,k}M_kP_{d,k},\tau_k\right)^\mc{U}\right).\] For $1 \leq d \leq m$, let $\varphi_d$ denote the projection onto the $d^\text{th}$ summand of \[\bigoplus_{d=1}^m\left(\left(P_{d,k}M_kP_{d,k},\tau_k\right)^\mc{U}\right).\] Since $\mc{A}_{v_0}$ is $\ms{C}$-tracially stable, for each $k \in \mb{N}$ there is a unital $*$-homomorphism \[\pi_{v_0,d,k}: \mc{A}_{v_0} \rightarrow P_{d,k}M_kP_{d,k}\] such that $(\pi_{v_0,d,k}(a))_\mc{U} = \varphi_d \circ \pi(a)$ for every $a \in \mc{A}_{v_0}$.  Define \[\pi_{v_0,k}:= \bigoplus_{d=1}^m \pi_{v_0,d,k}.\]  For each $k \in \mb{N}$, define \[\pi_{v_1,k}(a) = \sum_{d=1}^m \pi(a)(\omega_d)P_{d,k},\] and for $v \in V\Gamma\setminus \left\{v_0,v_1\right\}$, set \[\pi_{v,k} = \rho_{v,k}.\] By construction, for each $k \in\mb{N}$, the images of the $\pi_{v,k}$'s satisfy the commuting relations prescribed by $\Gamma$. So for each $k \in \mb{N}$, we can define \[\pi_k = \bigstar_\Gamma \pi_{v,k}.\] We claim that these $\pi_k$'s satisfy \eqref{tol}.  For $1 \leq j \leq n$ we have
\begin{align}
\|(\pi_k(x_j))_\mc{U} - \pi(x_j)\|_2 &\leq 2\varepsilon + \left|\left|\sum_{i=1}^{N_j}\left(\left(\pi_k\left(w_i^{(j)}\right)\right)_\mc{U} - \pi\left(w_i^{(j)}\right)\right)\right|\right|_2\notag\\
&< 2\varepsilon + \sum_{i=1}^{N_j}\left|\left|\left(\left(\pi_k\left(w_i^{(j)}\right)\right)_\mc{U} - \pi\left(w_i^{(j)}\right)\right)\right|\right|_2\notag\\
& < 2\varepsilon + \sum_{i = 1}^{N_j} S_i^{(j)} \frac{\varepsilon}{NS}\label{why}\\
&\leq3\varepsilon.\notag
\end{align}
Line \eqref{why} follows from a straightforward exercise applying \eqref{est} and iterating the following standard approximation technique.
\begin{align*}
||xyzyw - xy'zy'w||_2 & = ||xyzyw - xy'zyw + xy'zyw - xy'zy'w||_2\\
&\leq ||xyzyw - xy'zyw||_2 + ||xy'zyw - xy'zy'w||_2\\
&\leq ||x||\cdot ||y-y'||_2 \cdot ||zyw|| + ||xy'z||\cdot ||y-y'||_2 \cdot ||w||
\end{align*}
This completes the proof.
\end{proof}

In light of Proposition \ref{gpprop}, we record the analogous group-theoretic result (again, using the terminology from \cite{beclub}) as follows.

\begin{cor}
Let $\Gamma \in \ms{G}^{(\infty)}$ be a pincushion class graph.  For each $v \in V\Gamma$, let $G_v$ be a countable discrete group satisfying the following properties.
\begin{enumerate}
\item If $v \in V\Gamma$ is a pin not adjacent to another pin, then $G_v$ is HS-stable;
\item If $v \in V\Gamma$ is isolated, then $G_v$ is HS-stable.
\item If $v \in V\Gamma$ fits neither of the above descriptions, then $G_v$ is abelian.
\end{enumerate}
Then $\bigstar_\Gamma G_v$ is HS-stable.  In particular, the graph product of abelian groups over a pincushion class graph is HS-stable.
\end{cor}

\section{A selective version of Lin's Theorem}\label{sellin}

Lin's Theorem (\cite{lin2}) resolved a long-standing problem posed by Halmos in \cite{halmos} and can be stated as follows.

\begin{thm}[\cite{lin2}]
For any $\varepsilon > 0$ there is a $\delta > 0$ such that for any $n \in \mb{N}$ and any two contractive self-adjoint elements $a,b \in \mb{M}_n$ with $||[a,b]|| < \delta$ there exist commuting self adjoints $a',b' \in \mb{M}_n$ with $||a - b'|| + ||a - b'|| < \varepsilon$.  Here, $\mb{M}_n$ denotes the $n\times n$ matrices with complex entries, and $[x,y] := xy-yx$.
\end{thm}

\noin In English, this says \tql a pair of almost commuting self-adjoint matrices is near a pair of commuting self-adjoint matrices.\tqr This is a widely celebrated result in the operator theory community and frequently has been the subject of extension and variation.  Friis-R\o rdam in \cite{friror} extended Lin's Theorem to more operator algebras, including factor von Neumann algebras.  Further quantitative refinements can be found in \cite{filsaf, kacsaf}. A number of counterexamples are also present throughout the literature; in particular, it has been shown that the above statement does not apply to unitaries (\cite{voiccomm}), or to more than two self-adjoints (\cite{voicheis, davicomm, choicomm, loringcomm1, loringcomm2}). 

Halmos's problem is often posed in different contexts.  When we consider the Hilbert-Schmidt norm given by a trace, we have a positive result for an $n$-tuple of normal, self-adjoint, or unitary elements.  To put a finer point on it, Hadwin gave the following Hilbert-Schmidt variation of Lin's Theorem in \cite{hadwin}.

\begin{thm}[\cite{hadwin}]\label{hadwin}
For every $n \in \mb{N}$ and each $\varepsilon > 0$, there is a $\delta>0$ such that for any finite factor von Neumann algebra $M$ with tracial state $\tau$, if $a_1,\dots, a_n$ are contractions in $M$ such that \[||[a_i,a_i^*]||_{2,\tau}, ||[a_i,a_j]||_{2,\tau}< \delta\] for $1 \leq i,j \leq n$, then there is a commuting family $\left\{b_1,\dots, b_n\right\}$ of normal elements in $M$ such that \[\sum_{i = 1}^n ||a_i - b_i||_{2,\tau} \leq \varepsilon.\] If $||a_i-a_i^*||_{2,\tau} < \delta$ for $1 \leq i \leq n$, then the $a_i$'s can be taken to be self-adjoint; and if $||1 - a_ia_i^*||_{2,\tau} < \delta$, then the $b_i$'s can be taken to be unitaries.  Here, $||x||_{2,\tau} = \sqrt{\tau(x^*x)}$.
\end{thm}

\noin See \cite{glebsky, filkac} for more on the Hilbert-Schmidt version. Theorem \ref{hadwin} is proved in \cite{hadwin} using an ultraproduct argument. Upon revisiting Hadwin's proof of Theorem \ref{hadwin} in \cite{hadwin} one can see how it utilizes the  notion of tracial stability two decades before its definition. 
It should also noted that Arzhantseva-P\u{a}unescu recently proved an analogous result for symmetric groups in \cite{arzpau}: with respect to the Hamming metric, almost commuting permutations are near commuting permutations.  More recently, Elek-Grabowski proved an analogous result for unitary or self-adjoint matrices with respect to the rank metric in \cite{elegra}.

The main result of this section is a selective version of Lin's Theorem.  As in \cite{hadwin, glebsky, filkac} we consider the Hilbert-Schmidt norm from a tracial state; so a more accurate description would be a selective version of Hadwin's tracial variation of Lin's Theorem.  We show that the statement of Theorem \ref{hadwin} holds when the almost-commuting hypothesis is assumed only for a selected sub-collection of pairs of the $a_i$'s.  Using graph products, the selectivity of the almost commuting relations can be encoded using a graph with the vertices representing the $n$ elements and the edges connecting the selected pairs of elements which almost commute. In particular, we apply Theorem \ref{pcts} to show that for graphs in the \emph{pincushion class}, a version of Theorem \ref{hadwin} holds for elements that almost commute according to those graphs.  

\begin{thm}\label{mr}
Let $\Gamma \in \ms{G}^{(\infty)}$ be a pincushion graph, and let $\varepsilon >0$ be given.  Then there exists $\delta>0$ such that for any $C^*$-algebra $M \in \ms{C}$ with tracial state $\tau$, if for every $v \in V\Gamma, a_v$ is a contraction in $M$ such that $\left|\left|[a_v,a_v^*]\right|\right|_{2,\tau}< \delta$ for every $v \in V\Gamma$ and \[\left|\left|[a_v,a_w]\right|\right|_{2,\tau} < \delta\] whenever $(v,w) \in E\Gamma$, then there is a family $\left\{b_v\right\}_{v \in V\Gamma}$ of normal elements in $M$ commuting according to $\Gamma$ such that \[\sum_{v \in V\Gamma} ||a_v - b_v||_{2,\tau} \leq \varepsilon.\] If $||a_v-a_v^*||_{2,\tau} < \delta$ for $v \in V\Gamma$, then the $b_v$'s can be taken to be self-adjoint; and if $||1 - a_va_v^*||_{2,\tau} < \delta$, then the $b_v$'s can be taken to be unitaries.
\end{thm}

\begin{proof}
As in the proof of Theorem \ref{hadwin}, we will apply tracial stability. Let $\Gamma \in \ms{G}^{(\infty)}$ and $\varepsilon > 0$ be given.  Suppose for the sake of contradiction that for every $k \in \mb{N}$, there is a $C^*$-algebra $M_k \in \ms{C}$ with tracial state $\tau_k$ and elements $\left\{a_{v,k}\right\}_{v\in V}$ such that $\ds ||[a_{v,k},a_{v,k}^*]||_{2,\tau_k} < \frac{1}{k}$ for every $v \in V$, and $\ds ||[a_{v,k},a_{w,k}]||_{2,\tau_k}< \frac{1}{k}$ whenever $(v,w) \in E\Gamma$ with the property that for any family of normal elements $\left\{b_{v,k}\right\}_{v\in V\Gamma} \subset M_k$ commuting according to $\Gamma$, we have 
\begin{align}
\sum_{v \in V\Gamma} ||a_{v,k} - b_{v,k}||_{2,\tau_k} > \varepsilon. \label{gg}
\end{align} Now consider the ultraproduct $(M_k,\tau_k)^\mc{U}$.  For $v \in V\Gamma$, let $x_v = (a_{v,k})_\mc{U}$.  Then the elements $\left\{x_v\right\}_{v \in V\Gamma}$ are normal in $(M_k,\tau_k)^\mc{U}$ and commute according to $\Gamma$.  For each $v \in V$ consider the identity $*$-homomorphism denoted $\pi_v: C^*(1, x_v) \rightarrow (M_k,\tau_k)^\mc{U}$.  Since the images of the $\pi_v$'s commute according to $\Gamma$, we can form the $*$-homomorphism $\pi = \bigstar_\Gamma \pi_v: \bigstar_\Gamma C^*(1, x_v) \rightarrow (M_k,\tau_k)^\mc{U}$.  Since $C^*(1, x_v)$ is commutative for every $v \in V$ and $\Gamma$ is in the pincushion class, $\pi$ is approximately liftable by Theorem \ref{pcts}.  Hence, for each $k \in \mb{N}$, there is a $*$-homomorphism $\pi_k: \bigstar_\Gamma C^*(1, x_v) \rightarrow M_k$ such that $(\pi_k(a))_\mc{U} = \pi(a)$ for every $a \in \bigstar_\Gamma C^*(1, x_v)$.  For each $v \in V$ and each $k \in \mb{N}$, let $b_{v,k} = \pi_k(x_v)$. Then for each $k \in \mb{N}$, $\left\{b_{v,k}\right\}_{v \in V}$ is a collection of normal elements in $M_k$ commuting according to $\Gamma$ with $(a_{v,k})_\mc{U} = (b_{v,k})_\mc{U}$ for every $v \in V$.  This contradicts \eqref{gg}.
\end{proof}

\begin{rmk}
Note that if $\Gamma$ is taken to be a complete graph, then the above proof recovers the spirit of the proof of Theorem \ref{hadwin} in \cite{hadwin}.  This shows that Hadwin's usage of tracial stability predates its inception in \cite{hadshu}.
\end{rmk}

\section{Approximation properties of RAAGs}\label{app}

This section lays out the narrative of how Theorem \ref{pcts} impacts certain approximation properties of right-angled Artin groups (see Definition \ref{raagdfn}). 

To begin, we show that if $A$ is a right-angled Artin group (RAAG), its full group $C^*$-algebra $C^*(A)$ is quasidiagonal (QD) by extending Brown-Ozawa's proof of the fact that $C^*(\mb{F}_n\times \mb{F}_n)$ is QD (Proposition 7.4.5 in \cite{brownozawa}). Quasidiagonality is an approximation property for $C^*$-algebras that has been the subject of intense study for several decades now.  Group $C^*$-algebras are natural examples for which one can consider quasidiagonality.  In the context of the reduced $C^*$-algebra of a countable discrete group $G$, the score is settled: $C^*_r(G)$ is QD if and only if $G$ is amenable.  The \tql only if\tqr direction has been know for over thirty years now thanks to Rosenberg in \cite{hadros}, whereas the \tql if\tqr direction (Rosenberg's conjecture) was only recently resolved by Tikuisis-White-Winter in \cite{tww}.  For full/universal $C^*$-algebras of (non-amenable) groups the question of quasidiagonality is not as well-understood, but there are some surprising results nonetheless.  For example, due to Choi's residual finite dimensional result in \cite{choi}, $C^*(\mb{F}_n)$ is QD.  We already mentioned Brown-Ozawa's result showing that $C^*(\mb{F}_n\times \mb{F}_n)$ is QD.  Since these groups are examples of RAAGs, Theorem \ref{raagsqd} generalizes these two examples.  In the wake of the resolution of Rosenberg's conjecture in \cite{tww}, it is worth recording that if we combine the fact that amenable groups have QD $C^*$-algebras with the fact that the universal free product of QD $C^*$-algebras is QD (Proposition 13 of \cite{boca2}), then we obtain the following theorem.

\begin{thm}\label{fpgpqd}
Let $\left\{G_i\right\}_{i \in I}$ be a collection of countable discrete amenable groups.  Then $C^*(*_{i \in I} G_i)\cong *_{i \in I}C^*(G_i)$ is QD.
\end{thm}

\noin Theorems \ref{fpgpqd} and \ref{raagsqd} add to the list of groups with QD universal $C^*$-algebras.

RAAGs were first introduced by Baudisch in \cite{baudisch}.  Since their introduction, such groups have been heavily studied in the group theory literature--see \cite{charney} for a survey and repository of references.  In addition to the attention earned by their subgroups, the relation of RAAGs with special cube complexes in Agol's resolution of the virtual Haken conjecture has further enriched their cachet--see \cite{agol, hagwis, wise}. 
\begin{dfn}\label{raagdfn}
A group $A$ with presentation \begin{align} A= \langle s_1,\dots, s_n | s_is_j = s_j s_i \text{ if } (i,j) \in E \rangle \label{raagdef}\end{align} for some symmetric subset $E\subset \left\{1,\dots, n\right\}^2$, is called a \emph{right-angled Artin group (or RAAG)}. 
\end{dfn}
\noin A RAAG can be perceived as a graph product of copies of $\mb{Z}$; that is, $A = \bigstar_\Gamma \mb{Z}$ where $V\Gamma = \left\{1,\dots, n\right\}$ and $E\Gamma = E$ for $E$ as in \eqref{raagdef}.  For instance, $\mb{F}_2 \times \mb{F}_2 \cong \bigstar_\Gamma \mb{Z}$ where $\Gamma$ is the following graph.

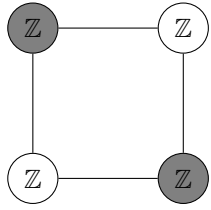
\begin{center}
\begin{tikzpicture}
\path (0,0) node [draw,shape=circle,] (p0) {$\mb{Z}$}
(0,2) node [draw,shape=circle,fill=gray]  (p1) {$\mb{Z}$}
(2,2) node [draw,shape=circle,] (p2) {$\mb{Z}$}
(2,0) node [draw,shape=circle,fill=gray] (p3) {$\mb{Z}$};
\draw (p0) -- (p1)
(p1) -- (p2)
(p2) -- (p3)
(p3) -- (p0);
\end{tikzpicture}
\end{center}
Here, the unshaded vertices generate one copy of $\mb{F}_2$ and the shaded vertices generate the second copy commuting with the first.



The following proposition shows that graph products are durable under the $C^*(\bullet)$ functor from groups to $C^*$-algebras.

\begin{prop}[\cite{casfim}]
$C^*(\bigstar_\Gamma G_v) \cong \bigstar_\Gamma C^*(G_v)$.
\end{prop}

\noin Thus, given a RAAG $A$, we have that $C^*(A) = C^*(\bigstar_\Gamma \mb{Z}) \cong \bigstar_\Gamma C^*(\mb{Z})$.  We record the following well-known fact about these component $C^*$-algebras $C^*(\mb{Z}) \cong C(\mb{T})$.

\begin{prop}
$C^*(\mb{Z})$ enjoys the following universal property.  Given a Hilbert space $\mc{H}$ and a unitary $u \in B(\mc{H})$, there exists a unital $*$-homomorphism $k: C^*(\mb{Z}) \rightarrow C^*(u)$ such that $k(s) = u$ where $s$ is the generator of $\mb{Z}$ and $C^*(u)$ denotes the $C^*$-algebra generated by $u$.
\end{prop}

\begin{dfn} A $C^*$-algebra $\mc{A}$ is \emph{quasidiagonal (or QD)} if for every finite subset $F \subset \mc{A}$ and every $\varepsilon >0$ there exist $n \in \mb{N}$ and a contractive completely positive map $\varphi: \mc{A} \rightarrow \mb{M}_n$ such that \[||\varphi(ab) - \varphi(a)\varphi(b)|| < \varepsilon\] and \[||\varphi(a)|| > ||a||-\varepsilon\] for every $a,b \in F$. We will say that a group $G$ is QD if $C^*(G)$ is QD. So, for instance, Theorem \ref{fpgpqd} says that the free product of amenable groups is QD.
\end{dfn}

The last ingredient we need before we give the next main result is a version of Voiculescu's homotopy invariance property for quasidiagonality from \cite{voicqd}.  

\begin{dfn}
Let $\mc{A}$ and $\mc{B}$ be two $C^*$-algebras. Two $*$-homomorphisms $\varphi,\psi: \mc{A} \rightarrow \mc{B}$ are \emph{homotopic} if there are $*$-homomorphisms $\sigma_t: \mc{A} \rightarrow \mc{B}$ for $t \in [0,1]$ such that $\sigma_0 = \varphi, \sigma_1 = \psi$, and for every $a \in \mc{A}$, $\sigma_t(a)$ is a norm-continuous path.  
\end{dfn}

\begin{prop}[\cite{brownozawa}]
Let $\varphi, \psi: \mc{A} \rightarrow \mc{B}$ be homotopic $*$-homomorphisms such that $\varphi$ is injective and $\psi(\mc{A})$ is QD.  Then $\mc{A}$ is QD.
\end{prop}


\begin{thm}\label{raagsqd}
Let $A$ be a RAAG.  Then $A$ is QD.
\end{thm}

\begin{proof}
As mentioned above, this proof extends the proof of Proposition 7.4.5 in \cite{brownozawa}.  Say $A = \bigstar_\Gamma \mb{Z}$ with $V\Gamma = \left\{1, \dots, n\right\}$, and let $s_1, \dots, s_n$ be the generators for $A$ as in \eqref{raagdef}, identifying them with the generating unitaries in $C^*(A)$. Fix a faithful $*$-representation $\pi: C^*(A) \rightarrow B(\mc{H})$.  It is obvious that the commuting relations prescribed by $\Gamma$ for $s_1,\dots, s_n$ extend to the generated abelian von Neumann algebras $\pi(C^*(s_1))'',\dots, \pi(C^*(s_n))''$. Let $\sigma: C^*(A) \rightarrow B(\mc{H})$ denote the trivial representation; that is, $\sigma(s_i) = 1_\mc{H}$ for every $1 \leq i \leq n$. Now, for each $1 \leq i \leq n,$ there is a norm-continuous path of unitaries in $\pi(C^*(s_i))''$ from $\pi(s_i)$ to $1_\mc{H}$.  Let $\pi_i = \pi|_{C^*(s_i)}: C^*(s_i) \rightarrow \pi(C^*(s_i))''$ and $\sigma_i= \sigma|_{C^*(s_i)}: C^*(s_i) \rightarrow \pi(C^*(s_i))''$. Thanks to the universal property for $C^*(s_i) \cong C^*(\mb{Z})$, we have that $\pi_i$ and $\sigma_i$ are homotopic for $1 \leq i \leq n$. The universal property for graph products of $C^*$-algebras then tells us that since $\pi(C^*(s_1))'',\dots, \pi(C^*(s_n))''$ commute according to $\Gamma, \pi = \bigstar_\Gamma \pi_i$ and $\sigma = \bigstar_\Gamma \sigma_i$ are homotopic $*$-homomorphisms.  Since $\sigma(C^*(A)) = \mb{C}$ is QD, then by the above proposition, so is $\pi(C^*(A)) \cong C^*(A)$.  
\end{proof}

\noin Since subgroups of RAAGs are of major interest, we observe the following immediate corollary.

\begin{cor}
Let $A$ be a RAAG and $B \leq A$ be a subgroup. Then $B$ is QD.
\end{cor}

\begin{proof}
This follows from the facts that $C^*(B) \subset C^*(A)$ and quasidiagonality is preserved under taking subalgebras.
\end{proof}

A consequence of quasidiagonality is the existence of an amenable trace, defined as follows.

\begin{dfn}\cite{invar}\label{amendef}
Let $\mc{A}$ be a unital $C^*$-algebra, and let $T(\mc{A})$ denote the collection of all tracial states on $\mc{A}$.  A trace $\tau \in T(\mc{A})$ is \emph{amenable} if there exists a sequence of u.c.p. maps $\varphi_k: \mc{A} \rightarrow \mb{M}_{n(k)}$ such that $\ds \lim_{k\rightarrow \infty} ||\varphi_k(ab) - \varphi_k(a)\varphi_k(b)||_{2,\text{tr}_{n(k)}} = 0$ and $\ds \tau(a) = \lim_{k\rightarrow \infty} \text{tr}_{n(k)} \circ \varphi_k(a)$ for every $a,b \in \mc{A}$.  Let $\text{AT}(\mc{A})$ denote the set of amenable traces on $\mc{A}$.
\end{dfn}

\noin So clearly we have the following corollary.

\begin{cor}
Let $A$ be a RAAG and $B \leq A$ be a subgroup.  Then $C^*(B)$ has a (non-trivial) amenable trace.
\end{cor}

Now that we know the set of amenable traces on a RAAG is nonempty, we wish to further analyze such amenable traces on RAAGs.  In \cite{invar}, Brown studies amenable traces and specific subclasses of amenable traces with stronger approximation properties.  A subclass of interest for this article is the following.

\begin{dfn}[\cite{invar}]\label{lfddef}
A trace $\tau \in T(\mc{A})$ is called \emph{locally finite dimensional} if there exist u.c.p maps $\varphi_k: \mc{A} \rightarrow \mb{M}_{n(k)}$ such that $\text{tr}_{n(k)} \circ \varphi_k \rightarrow \tau$ in the weak-$*$ topology and $\ds \lim_{k\rightarrow \infty} d(a, \mc{A}_{\varphi_k}) = 0$ for every $a \in \mc{A}$.  Here $\ds d(a, \mc{A}_{\varphi_k}) = \inf_{b \in \mc{A}_{\varphi_k}} ||a - b||$ and $\mc{A}_{\varphi_k}$ is the multiplicative domain of $\varphi_k$.  Let $\text{AT}_\text{LFD}(\mc{A})$ denote the set of locally finite dimensional traces on $\mc{A}$.
\end{dfn}

Consider the following fact for free groups.

\begin{thm}[\cite{invar}]\label{fnlfd}
$\text{AT}(C^*(\mb{F}_n)) = \text{AT}_\text{LFD}(C^*(\mb{F}_n))$.
\end{thm} 

\noin This result also clearly holds for any abelian RAAG (i.e. $\mb{Z}^n$); thus it is reasonable to ask if such a result holds for any RAAG.  Theorem \ref{fnlfd} is proved in \cite{invar} essentially by using the fact that $C^*(\mb{F}_n)$ is tracially stable for the class of finite factor von Neumann algebras.  Thus, applying Theorem \ref{pcts} yields the following generalization of Theorem \ref{fnlfd} that includes much more than just pincushion RAAGs--we thank Pieter Spaas for suggesting that the proof applies to such a general setting.

\begin{thm}
If $\mc{A}$ is a unital separable \textbf{matrix}-tracially stable $C^*$-algebra, then $\text{AT}(\mc{A}) = \text{AT}_\text{LFD}(\mc{A})$.
\end{thm}

\begin{proof}
We generalize the argument for the proof of Proposition 4.1.14 in \cite{invar}.  Let $\tau \in \text{AT}(\mc{A})$.  By Theorem 3.1.7 of \cite{invar}, there is a $*$-homomorphism $\pi: \mc{A} \rightarrow R^\mc{U}$ ($R$ denotes the separably acting hyperfinite II$_1$-factor) such that $\tau = \tau_\mc{U} \circ \pi$.  By standard approximation arguments, we can assume that $\pi$ is a $*$-homomorphism into the ultraproduct $(\mb{M}_{n(k)})^\mc{U} \subset R^\mc{U}$ where $\mb{M}_{n(k)}$ is a unital matrix subalgebra of $R$.  Since $\mc{A}$ is \textbf{matrix}-tracially stable, we have that \[\pi: \mc{A} \rightarrow (\mb{M}_{n(k)})^\mc{U} \subset R^\mc{U}\] is approximately liftable.  So there is a sequence of $*$-homomorphisms $\pi_k: \mc{A}  \rightarrow \mb{M}_{n(k)}$ such that $(\pi_k(a))_\mc{U} = \pi(a)$ for every $a \in \mc{A}$.  Then it is clear to see that $\tau$ satisfies Definition \ref{lfddef}.
\end{proof}

\begin{cor}
If $\Gamma \in \ms{G}^{(\infty)}$, then $\text{AT}(C^*(\bigstar_\Gamma \mb{Z})) = \text{AT}_\text{LFD}(C^*(\bigstar_\Gamma \mb{Z}))$.
\end{cor}

\begin{proof}
This follows from the fact that \[C^*(\bigstar_\Gamma \mb{Z}) \cong \bigstar_\Gamma C^*(\mb{Z}) \cong \bigstar_\Gamma C(\mb{T})\] is \fvna-tracially stable by Theorem \ref{pcts}.
\end{proof}

While the pincushion class of finite simplicial graphs is a large subclass that includes complete graphs, trees, and many more, it is still of significant interest to extend the results of this article to all finite simplicial graphs. For instance, the smallest non-pincushion graph is the square.
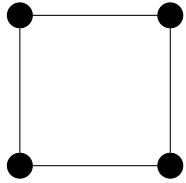
\begin{center}
\begin{tikzpicture}
\path (0,0) node [draw,shape=circle,fill=black] (p0) {}
(0,2) node [draw,shape=circle,fill=black]  (p1) {}
(2,2) node [draw,shape=circle,fill=black] (p2) {}
(2,0) node [draw,shape=circle,fill=black] (p3) {};
\draw (p0) -- (p1)
(p1) -- (p2)
(p2) -- (p3)
(p3) -- (p0);
\end{tikzpicture}
\end{center}
Thus we wish to resolve the following question.

\begin{?}
Let $\ms{C}$ be a class of \rrz $C^*$-algebras closed under taking direct sums and unital corners, and let $\Gamma$ be a finite simplicial graph.  For each $v \in V\Gamma$, let $X_v$ be a compact Hausdorff space; is $\bigstar_\Gamma C(X_v)$ $\ms{C}$-tracially stable?
\end{?}

\bibliographystyle{plain}
\bibliography{gpqdbib}{}

\end{document}